\documentclass[11pt]{article}

\usepackage{amsthm}
\usepackage{amssymb}
\usepackage{mathtools}
\usepackage{mathrsfs}
\usepackage[numbers]{natbib}
\usepackage{enumitem}
\usepackage[margin=1in]{geometry}
\usepackage{titlesec}
\usepackage{chngcntr}
\usepackage{subcaption}

\titleformat{\section}
{\centering\large\scshape}{\thesection.}{.5em}{}

\newtheorem{thm}{Theorem}[section]
\newtheorem{lem}[thm]{Lemma}
\newtheorem{prop}[thm]{Proposition}
\newtheorem{cor}[thm]{Corollary}

\theoremstyle{definition}

\theoremstyle{remark}

\newcommand{\Z}{\mathbb{Z}}

\newcommand{\R}{\mathbb{R}}
\newcommand{\C}{\mathbb{C}}

\newcommand{\pth}[1]{\left( #1 \right)}

\newcommand{\hor}[1]{\left[ #1 \right)}

\newcommand{\set}[1]{\left\{ #1 \right\}}
\newcommand{\abs}[1]{\left| #1 \right|}

\newcommand{\floor}[1]{\left\lfloor #1 \right\rfloor}

\newcommand{\sumabs}[1]{\bigg| #1 \bigg|}
\newcommand{\sumpth}[1]{\bigg( #1 \bigg)}
\newcommand{\fracp}[2]{\pth{\frac{#1}{#2}}}
\newcommand{\dfracp}[2]{\pth{\dfrac{#1}{#2}}}
\newcommand{\fract}[2]{{\textstyle\frac{#1}{#2}}}
\newcommand{\mand}{\qquad\text{and}\qquad}

\renewcommand{\mod}[1]{\ (\textnormal{mod}\ #1)}

\newcommand{\ep}{\varepsilon}
\renewcommand{\Re}{\mathrm{Re}\phantom{.}}
\renewcommand{\Im}{\mathrm{Im}\phantom{.}}

\renewcommand{\phi}{\varphi}

\counterwithin*{equation}{section}

\newcommand{\cola}{\mathcal{A}}\newcommand{\colb}{\mathcal{B}}\newcommand{\colc}{\mathcal{C}}\newcommand{\cold}{\mathcal{D}}\newcommand{\coll}{\mathcal{L}}\newcommand{\coln}{\mathcal{N}}\newcommand{\cols}{\mathcal{S}}

\newcommand{\ida}{\mathfrak{a}}\newcommand{\idp}{\mathfrak{p}}

\usepackage{tikz}
\usepackage{pgfplots}
\usetikzlibrary{calc}
\pgfplotsset{compat=1.7}

\begin{document}
\bibliographystyle{abbrv}

\begin{center}
{\large\bf GAUSSIAN PRIMES IN NARROW SECTORS}\\[1.5em]
{\scshape Joshua Stucky}
\end{center}
\vspace{.5em}

\section{Introduction}

\indent A classical result of Huxley \citep{Huxley} states that for sufficiently large $x$ and any $\theta > 7/12$, the interval $[x,x+x^{\theta}]$ contains a rational prime. In this paper, we investigate an analogous problem about Guassian primes. To be precise, let $\phi\in\R$, $0 < \delta \leq \pi/2$, $0<\theta \leq 1$, and $x$ large. We are interested in the cardinality of the set
\[
\set{a+bi\in\Z[i]: (a+bi)\ \text{is prime,}\ \phi < \arg(a+bi)\leq \phi+\delta,\ x-x^\theta < a^2+b^2 \leq x}.
\]
Here $(a+bi)$ denotes the ideal generated by $a+bi$. As is common in such problems, it is more convenient to count these ideals with a suitable weight. Denote by $\ida$ the ideal in $\Z[i]$ generated by $a+bi$ and by $N\ida = a^2+b^2$ its norm. If we define
\[
\Lambda(\ida) = \begin{cases}
\log N\ida & \text{if $\ida=\idp^m$ with $\idp$ prime and $m \geq 1$}, \\
0\phantom{====} & \text{otherwise},
\end{cases}
\]
then our problem translates to obtaining an asymptotic estimate for 
\[
\psi(x,y;\phi,\delta) = \sum_{\substack{x-y < N\ida \leq x\\ \phi < \arg \ida \leq \phi+\delta}} \Lambda(\ida).
\]
Ricci \citep{RicciThesis} has shown that for all $\ep > 0$ and $\delta \geq x^{-3/10+\ep}$, one has
\[
\psi(x,x;\phi,\delta) \sim \frac{2\delta x}{\pi}.
\]
We generalize this and prove the following
\begin{thm}\label{thm:Asymptotic}
For any $\ep > 0$, $\phi \in \R$, $x$ sufficiently large, $\theta > 7/10$, and $\delta x^\theta \geq x^{7/10+\ep}$, we have
\[
\psi(x,x^\theta,\phi,\delta)  \sim \frac{2\delta x^\theta}{\pi}.
\]
\end{thm}

Geometrically, the parameters $x,\theta,\phi,\delta$ describe a sector centered at the origin. The inner and outer radii of this sector are $\sqrt{x-x^\theta}$ and $\sqrt{x}$, and the sector is cut by rays emanating from the origin with angles $\phi$ and $\phi+\delta$. Ricci's result gives the expected number of prime ideals in a sector so long as the inner radius $\sqrt{x-x^\theta}$ is essentially 0 and the angle $\delta$ between the rays is sufficiently wide. Theorem \ref{thm:Asymptotic} claims the more general result that one obtains the expected number of prime ideals so long as the area of the sector is sufficiently large.\\

\noindent\textbf{A note on the literature.} It should be noted that Maknys \citep{MaknysSectors} has claimed a result similar to Theorem \ref{thm:Asymptotic}, but with the exponent $11/16$ in place of $7/10$. However, Heath-Brown \citep{HeathBrownReview} has found an error in Maknys' proof of this result. He states that Maknys' proof, when corrected, yields the exponent $(221+\sqrt{201})/320 = 0.7349...$.  However, the result is potentially worse than $0.7349...$ because Maknys's proof depends on a zero density estimate (Theorem 2 of \citep{MaknysZeros}), the proof of which also contains an error. In particular, there is an incorrect application of Theorem 1 of \citep{MaknysAux}. For a version of Theorem 1 of \citep{MaknysAux} which is applicable in the proof of Maknys' zero density result, see Theorem 6.2 and the end of Section 7 of \citep{ColemanBinary}. \\

\noindent\textbf{Outline of the Paper} To orient the reader, we provide an outline of the paper. In Section \ref{sec:InitialDecomposition}, we begin by smoothing the angular and norm regions for $\psi(x,x^\theta,\phi,\delta)$, and then express these regions via a sum of Hecke characters $\lambda^m$ and an integral of $(N\ida)^{it}$. The main term in Theorem \ref{thm:Asymptotic} then arises from the contribution of the principal character. After applying an analogue of Heath-Brown's identity in $\Z[i]$ (see Lemma \ref{lem:HeathBrown} below), we are left to bound a sum of $O((\log x)^{2J+2})$ expressions roughly of the form
\[
\sum_{M\leq m\leq 2M} c_m \int_{T}^{2T} \tilde{V}\pth{\fract{1}{2}+it} \sum_{\substack{\ida=\ida_1\cdots\ida_{2J}\\ N\ida_j \asymp N_j}} a_1(\ida_1)\cdots a_{2J}(\ida_{2J}) \frac{\lambda^m(\ida)}{(N\ida)^{1/2+it}}\ dt
\]
for some parameters $N_i$. Here the $c_m$ are Fourier coefficients and $\tilde{V}$ is a Mellin transform. Using estimates for $c_m$ and $\tilde{V}$, this reduces to showing that
\[
\sum_{M\leq m\leq 2M} \int_{T}^{2T} \abs{F\pth{\fract{1}{2}+it}} dt \ll \frac{x^{1/2}}{(\log x)^A},
\]
where $F$ is the Dirichlet series appearing in the penultimate display.

In Section \ref{sec:Reduction}, we reduce this to bounding the number $R$ of pairs $m,t$ for which a particular factor $f$ of $F$ attains a large value. Specifically, for such a pair $m,t$, we have
\[
\sumabs{\sum_{N\ida \asymp N} c(\ida)\lambda^m(\ida)(N\ida)^{-it}} \gg W
\]
for some divisor-bounded coefficients $c(\ida)$ and $W>0$. In Section \ref{sec:MeanAndLarge}, we use mean- and large-value estimates to bound $R$. Specifically, we use a hybrid large sieve estimate due to Coleman and an analogue of Huxley's large value result, also due to Coleman. Writing $G = \sum \abs{c(\ida)}^2$, these yield
\[
\begin{aligned}
R &\ll NGW^{-2} + (M^2+T^2)GW^{-2}, \\
R &\ll NGW^{-2} + (M^2+T^2)NG^3W^{-6},
\end{aligned}
\]
respectively. We also use the ``trivial'' estimate
\[
R \ll \min(M,T)NGW^{-2} + MTGW^{-2},
\]
as well as a subconvexity result for the Hecke $L$-function, $L(s,\lambda^m)$, due to Ricci. There are a variety of ranges for $N,M,T$ to consider when deciding which estimate to use. This requires a case analysis which is done in Sections \ref{sec:ShortPolynomials} -- \ref{sec:LongPolynomials2}. Here we also indicate the ``worst cases'' of $N,M,T$ for which our estimates are sharp. 

We note that with an optimal large sieve, one would have the estimate
\begin{equation}\label{eq:OptimalLargeSieve}
R \ll NGW^{-2} + MTGW^{-2}.
\end{equation}
Although such a large sieve is not available in the literature, this would not improve our results (it would, however, simplify the case analysis). This is because one of the worst cases in our analysis remains a worst case when using this estimate. See Section \ref{sec:Conclusion} for this discussion.  \\

\noindent {\scshape\bf Acknowledgments.} I would like to thank my advisor, Xiannan Li, for suggesting this problem to me and for many helpful comments in the development of these results, as well as the referee for a number of useful suggestions to make this paper more elegant, readable, and transparent.

\section{Notation and Preliminary Lemmas}

We collect here some additional notation and lemmas we will need throughout the paper. The symbols $o, O,\ll,\gg,\asymp$ have their usual meanings. The letter $\ep$ denotes a sufficiently small positive real number, while $A,B,C$ stand for an absolute positive constants, all of which may be different at each occurrence. For example, we may write 
\[
x^\ep \log x \ll x^\ep, \qquad (\log x)^B (\log x)^B \ll (\log x)^B
\]
Any statement in which $\ep$ occurs holds for each positive $\ep$, and any implied constant in such a statement is allowed to depend on $\ep$. The implied constants in any statement involving the letters $A,B,C$ are also allowed to depend on these variables.

Similar to $\Lambda(\ida)$, we define
\[
\mu(\ida) = \begin{cases}
(-1)^r& \text{if $\ida=\idp_1\cdots\idp_r$ with $\idp_i$ distinct primes}, \\
0\phantom{====} & \text{otherwise}.
\end{cases}
\]
Let $\arg\ida$ be the argument of any one of the generators of $\ida$ (which is unique mod $\pi/2$). For $m\in\Z$, we define the angular Hecke characters
\[
\lambda^m(\ida) = e^{4im\arg\ida} = \fracp{\alpha}{\abs{\alpha}}^{4m},
\]
which are primitive with conductor $(1)$. Note that the character is well-defined since the particular generator $\alpha$ chosen for the definition above is immaterial. From these we get the Hecke $L$-functions, defined for $\Re s> 1$ by
\[
L(s,\lambda^m) = \sum_{\ida} \frac{\lambda^m(\ida)}{(N\ida)^s}.
\]
Here the sum is over all nonzero ideals of $\Z[i]$. These $L$-functions are absolutely convergent for $\Re s > 1$, and Hecke showed that, for $m\neq 0$, they have analytic continuation to all of $\C$ and satisfy a functional equation. We also have
\[
\frac{1}{L(s,\lambda^m)} = \sum_{\ida} \frac{\lambda^m(\ida)\mu(\ida)}{(N\ida)^s}, \qquad -\frac{L'(s,\lambda^m)}{L(s,\lambda^m)} = \sum_{\ida} \frac{\lambda^m(\ida)\Lambda(\ida)}{(N\ida)^s},
\]
which are also absolutely convergent for $\Re s > 1$. We summarize these facts in the following

\begin{lem}\label{lem:FunctionalEquation}
The function $L(s,\lambda^m)$ satisfies the functional equation
\begin{equation}\label{eq:FunctionalEquation}
L(s,\lambda^m) = \gamma(s,\lambda^m)L(1-s,\lambda^m),
\end{equation}
where
\[
\gamma(s,\lambda^m) = \pi^{2s-1} \frac{\Gamma(1-s+2\abs{m})}{\Gamma(s+2\abs{m})}.
\]
If $m \neq 0$, then $L(s,\lambda^m)$ is entire, and otherwise it is meromorphic with a simple pole at $s=1$ with residue $\frac{\pi}{4}$. We also have
\begin{equation}\label{eq:LSymmetry}
L(s,\lambda^m) = L(s,\lambda^{-m}).
\end{equation}
\end{lem}

These results are standard. See \citep{ANT}, for instance. We will need several results on the behavior of these functions in the critical strip. These are given in the following pair of lemmas.

\begin{lem}\label{lem:ZeroFreeRegion}
Let $V = (4m^2 + t^2)^{1/2}$. Then there exist absolute constants $C,\delta>0$ such that
\[
L(\sigma + it,\lambda^m) \ll V^{c(1-\sigma)^{3/2}}(\log V)^{2/3},
\]
uniformly for $1 - \delta < \sigma < 1$. It follows that there exists an absolute constant $C>0$ such that $L(s,\lambda^m)$ has no zeros in the region
\begin{equation}\label{eq:ZeroFree}
\sigma \geq 1 - C(\log V)^{-2/3}(\log\log V)^{-1/3}.
\end{equation}
\end{lem}

\begin{lem}\label{lem:CritStripEstimates}
For $\sigma$ in the region {\normalfont (\ref{eq:ZeroFree})}, we have
\[
\frac{L'(\sigma+it,\lambda^m)}{L(\sigma+it,\lambda^m)} \ll \log V, \qquad \frac{1}{L(\sigma+it,\lambda^m)} \ll \log V.
\]
\end{lem}

Lemma \ref{lem:ZeroFreeRegion} follows from Theorems 1 and 2 of \citep{ColemanZeroFree}, and the proof of Lemma \ref{lem:CritStripEstimates} follows closely the proof of Theorem 3.11 of \citep{Titchmarsh1986}. Next, we need an estimate for the number of lattice points in a suitably regular sector. 

\begin{lem}\label{lem:RegularLatticePointBound}
Let $\phi\in\R$, $x$ and $y$ be sufficiently large with $x^{1/2} \leq y \leq x$, and  $x^{-1/2}\leq \delta \leq \pi/2$. If
\[
\coln(x,y,\phi,\delta) = \#\set{a+bi\in\Z[i] : \phi\leq\arg(a+bi) \leq\phi+\delta,\ x-y\leq a^2+b^2 \leq x},
\]
then $\coln(x,y,\phi,\delta) \ll \delta y$.
\end{lem}

\begin{lem}\label{lem:DivisorBound}
Let $d_j(\ida)$ be the $j$-divisor function for $\Z[i]$. We have $d_j(\ida) \ll (N\ida)^\ep$, and for $y > x^{1/2}$ we also have
\[
\sum_{x - y < N\ida \leq x} d_j(\ida) \ll y(\log x)^{j-1}.
\]
For $\phi \in \R$ and $x^{-1/2} < \delta \leq \pi/2$, we also have
\[
\sum_{\substack{x-y< N\ida \leq x\\ \phi \leq \arg \ida \leq \phi+\delta}} d_j(\ida) \ll \delta y x^{\ep},
\]
The implied constants above depend only on $\ep$ and $j$. 
\end{lem}

The proof of Lemma \ref{lem:RegularLatticePointBound} is straightforward, and Lemma \ref{lem:DivisorBound} follows from  Shiu's work \citep{Shiu}. Our analysis makes use of an analogue of Heath-Brown's identity in $\Z[i]$ (see \citep{HeathBrown}). For technical reasons, it is more convenient to have a smoothed version of this identity. As such, let $W$ be a smooth function supported on $[\frac{1}{2},2]$ such that
\[
\sum_{n \geq 0} W(2^n t) = 1 \mand W^{j}(t) \ll t^{-j}
\]
for all $0< t \leq 1$. Then we have the following

\begin{lem}[Heath-Brown's Identity]\label{lem:HeathBrown}
Let $X > 1$ and $J$ be a positive integer, and let $W$ be as above. Then for any $\ida$ with $N\ida \leq X^J$, we have
\[
\begin{aligned}
\Lambda(\ida) &= \sum_{j=1}^{J} \binom{J}{j}(-1)^{j-1} \sum_{\ida_1\cdots \ida_{2J}=\ida} \log (\ida_1) \mu(\ida_{J+1})\cdots \mu(\ida_{2J}) \\
&\phantom{====}\bigtimes\sum_{\substack{n_1,\ldots,n_{j}\geq 0\\ n_{J+1},\ldots,n_{J+j}\geq 0}} W\fracp{N\ida_{1}}{X^J/2^{n_1}} \cdots W\fracp{N\ida_{j}}{X^J/2^{n_j}}W\fracp{N\ida_{J+1}}{X/2^{n_{J+1}}}\cdots W\fracp{N\ida_{J+j}}{X/2^{n_{J+j}}} \\
&\phantom{====}\bigtimes W(N\ida_{j+1})\cdots W(N\ida_{J}) W(N\ida_{J+j+1}) \cdots W(N\ida_{2J}).
\end{aligned}
\]

\end{lem}
Note that the terms on the last line simply force the ideals $\ida_{j+1},\ldots$ to have norm 1. The point of the lemma is that for $N\ida \leq X^J$, the function $\Lambda(\ida)$ can be decomposed into a linear combination of $O((\log X)^{2J})$ smooth sums of the form
\[
\sum_{\ida_1\cdots \ida_{2J}=\ida} \log (\ida_1) \mu(\ida_{J+1})\cdots \mu(\ida_{2J}) W\fracp{N\ida_1}{N_1} \cdots W\fracp{N\ida_{2J}}{N_{2J}},
\]
where $N_j = X^J/2^n$ or $X/2^n$ for some integer $n$, depending as $j\leq J$.

\section{Initial Decomposition}\label{sec:InitialDecomposition}
To estimate $\psi(x,x^\theta;\phi,\delta)$, we begin by smoothing the angular region for $\ida$. For this, we need 

\begin{lem}\label{lem:TestFunction}
Let $k\in\Z$ with $k\geq 0$ and let $\alpha,\beta,\Delta,L$ be real numbers satisfying
\[
L > 0, \qquad 0 < \Delta < \frac{L}{2}, \qquad \Delta \leq \beta-\alpha \leq L - \Delta.
\]
Then there exists an $L$-periodic function $P(x)$ with
\[
P(t) = \frac{1}{L}(\beta-\alpha) + \sum_{m\neq 0} c_m e^{4imt}
\]
which satisfies
\[
\begin{aligned}
P(t) &= 1 &&\text{if}\ t \in [\alpha,\beta], \\
P(t) &= 0 &&\text{if}\ t \in [\beta+\Delta,L+\alpha-\Delta] , \\
P(t) &\in [0,1] && \text{for all $t$,}
\end{aligned}
\]
and where the coefficients $c_m$ satisfy
\begin{equation}\label{eq:c_mBounds}
\abs{c_m} \leq \begin{cases}
\dfrac{1}{L}(\beta-\alpha), & \text{} \\[1em]
\dfrac{L}{\abs{m}} \dfracp{kL}{\Delta \abs{m}}^k & \text{if $m \neq 0$},
\end{cases}
\end{equation}
where the factor involving $k$ is taken to equal 1 when $k=0$.
\end{lem}

This result is classical. See, for example, Lemma A of Chapter 1, Section 2 of \citep{Karatsuba}. The special case $L = 1$ is proved there, but the lemma generalizes easily to arbitrary periods.

Let $P$ be as in the lemma with $L = \frac{\pi}{2}$, $\alpha = \phi$, $\beta = \phi+\delta$, and $\Delta = \delta x^{-\ep}$. Then
\[
\begin{aligned}
\psi(x,x^\theta;\phi,\delta) &= \sum_{x-x^\theta < N\ida \leq x} \Lambda(\ida) P(\arg \ida) + O\sumpth{\sum_{\substack{x-x^\theta < N\ida \leq x\\ \phi-\Delta\leq \arg\ida \leq \phi}} \Lambda(\ida)} + O\sumpth{\sum_{\substack{x-x^\theta < N\ida \leq x\\ \phi+\delta \leq \arg\ida \leq \phi+\delta+\Delta}} \Lambda(\ida)}.
\end{aligned}
\]
To estimate the error terms we note that the hypotheses of Theorem \ref{thm:Asymptotic} imply that $x^\theta \gg x^{7/10+\ep}$ and $\delta \gg x^{-3/10+\ep}$. In particular, we have
\[
x^{\theta} \geq  x^{1/2} \mand \Delta \geq x^{-1/2}.
\]
Since $\Lambda(\ida) \leq \log x$, we have by Lemma \ref{lem:RegularLatticePointBound} that
\[
\sum_{\substack{x-x^\theta < N\ida \leq x\\ \phi-\Delta\leq \arg\ida \leq \phi}} \Lambda(\ida) \ll (\log x) \coln(x,x^\theta,\phi-\Delta,\phi) \ll (\log x)x^\theta \Delta = o(\delta x^\theta),
\]
and similarly for the other error term. We expand $P(\arg\ida)$ using its Fourier series and write
\[
\sum_{x-x^\theta < N\ida \leq x} \Lambda(\ida) P(\arg \ida) = \sum_{x-x^\theta < N\ida \leq x} \Lambda(\ida) \sum_{m} c_m \lambda^m(\ida).
\]
We have
\[
\sum_{x-x^\theta < N\ida \leq x} \Lambda(\ida) = 2\sum_{\substack{x-x^\theta < p \leq x\\ p\equiv 1\mod{4}}} \log p + O(x^{1/2}\log x),
\]
(see, for instance, display 7.4 in Chapter 2 of \citep{RicciThesis} for this computation). Since $x^\theta \gg x^{7/12+\ep}$, the Siegel-Walfisz theorem in short intervals gives
\[
2\sum_{\substack{x-x^\theta < p \leq x\\ p\equiv 1\mod{4}}} \log p = x^\theta(1+o(1)),
\]
and since $c_0 = 2\delta \pi^{-1}$, we obtain
\[
\psi(x,x^\theta,\phi,\delta) = \frac{2\delta x^\theta}{\pi}(1+o(1)) + \sum_{x-x^\theta < N\ida \leq x} \Lambda(\ida) \sum_{m\neq 0} c_m \lambda^m(\ida).
\]
Using (\ref{eq:c_mBounds}), we truncate the Fourier series at $M_1$ to obtain
\[
\psi(x,x^\theta,\phi,\delta) = \frac{2\delta x^\theta}{\pi}(1+o(1)) + \sum_{x-x^\theta < N\ida \leq x} \Lambda(\ida) \sum_{1\leq \abs{m}\leq M_1} c_m \lambda^m(\ida) + O\sumpth{x^\theta \log x \fracp{\pi kx^\ep}{2\delta M_1}^k}
\]
for any $k \geq 1$. Choosing $M_1 = \delta^{-1} x^\ep$, a sufficiently large choice of $k$ depending only on $\ep$ makes the error term $o(\delta x^\theta)$, and so
\[
\psi(x,x^\theta,\phi,\delta) = \frac{2\delta x^\theta}{\pi}(1+o(1)) + \sum_{1\leq \abs{m}\leq M_1} c_m \sum_{x-x^\theta < N\ida \leq x} \Lambda(\ida) \lambda^m(\ida).
\]

Next, we smooth the norm-region for $\ida$. Let $V$ be a smooth function satisfying
\[
\begin{aligned}
V(t) &= 1 &&\text{if}\ t \in [x-x^\theta,x], \\
V(t) &= 0 &&\text{if}\ t \in \R\smallsetminus [x-x^\theta-x^{\theta-\ep},x+x^{\theta-\ep}], \\
V(t)&\in[0,1] &&\text{for all $t$.}
\end{aligned}
\]
Then $\tilde{V}$ satisfies
\begin{equation}\label{eq:VStripBound}
\tilde{V}(s) \ll x^{\theta+\sigma-1} \mand \tilde{V}(s) \ll \frac{x^{\sigma+(A-1)(1-\theta+\ep)}}{(1+\abs{t})^A}
\end{equation}
for any real $A \geq 1$, where the implied constant depends only on $A$ and $\sigma$. We obtain
\[
\psi(x,x^\theta,\phi,\delta) = \frac{2\delta x^\theta}{\pi}(1+o(1)) + \sum_{1\leq \abs{m}\leq M_1} c_m \sum_{\ida} \Lambda(\ida) V(N\ida) \lambda^m(\ida) = \frac{2\delta x^\theta}{\pi}(1+o(1)) + \cols,
\]
say, where the error in replacing the sharp cutoff with the smoothing function $V$ has been absorbed into the error term $o(\delta x^\theta)$.

We now employ Lemma \ref{lem:HeathBrown} with $X = (2x)^{1/J}$ for some integer $J \geq 1$ to be chosen. Then $\cols$ is a linear combination of $O((\log x)^{2J})$ sums of the form
\begin{equation}\label{eq:SigmaDef}
\begin{aligned}
S &= \sum_{1\leq \abs{m}\leq M_1} c_m \sum_{\ida=\ida_1\cdots\ida_{2J}} a_1(\ida_1)\cdots a_{2J}(\ida_{2J}) W_1(N\ida_1)\cdots W_{2J}(N\ida_{2J})\lambda^m(\ida) V(N\ida),
\end{aligned}
\end{equation}
where
\[
a_j(\ida) = \begin{cases}
\log N\ida & \text{if $j = 1$},\\
1 & \text{if $2\leq j \leq J$}, \\
\mu(\ida) & \text{if $J+1 \leq j \leq 2J$},
\end{cases}
\]
$W_{j}(k) = W(k/N_j)$, and $N_j = x/2^n$ or $X/2^n$ for some integer $n \geq 0$ depending as $j\leq J$.

It is natural to consider the Dirichlet series associated to the sums $S$. For each $j$ and $m$, put
\[
f_{j,m}(s) = \sum_{\ida} \frac{a_j(\ida)\lambda^m(\ida)W_j(N\ida)}{(N\ida)^s}
\]
and also let
\[
F_m(s) = \prod_{j=1}^{2J} f_{j,m}(s) = \sum_{\ida} \frac{a(\ida)\lambda^m(\ida)}{(N\ida)^s},
\]
where the coefficients satisfy
\[
\abs{a(\ida)} \ll d_{2J}(\ida)\log x.
\]
Then Mellin inversion gives
\[
S = \frac{1}{2\pi i} \int\limits_{(1/2)} \tilde{V}(s) \sum_{1\leq \abs{m}\leq M_1} c_m F_m(s)\ ds.
\]
For $\Re(s) = \frac{1}{2}$, we have
\[
F_m(s) \ll \log x\sum_{N\ida\leq 2x} \frac{d_{2J}(\ida)}{(N\ida)^{1/2}} \ll x^{1/2+\ep}.
\]
Also $\abs{c_m} \ll \delta$. Truncating the integral at height $T_1$ and using (\ref{eq:VStripBound}) then gives
\[
S = \frac{1}{2\pi i} \int_{1/2-iT_1}^{1/2+iT_1} \tilde{V}(s) \sum_{1\leq \abs{m}\leq M_1} c_m F_m(s)\ ds + O\pth{ x^{1/2+\ep}\frac{x^{1/2+(A-1)(1-\theta+\ep)}}{T_1^{A-1}}}
\]
for any $A \geq 1$. Choosing $T_1 = x^{1-\theta+\ep}$ and taking $A$ sufficiently large in terms of $\ep$ makes the error term negligible. We have $\abs{c_m} \ll \delta $ and $|\tilde{V}\pth{\frac{1}{2}+it}| \ll x^{\theta-1/2}$, so
\[
S \ll \frac{\delta x^\theta}{x^{1/2}} \sum_{1\leq \abs{m}\leq M_1} \int_{-T_1}^{T_1} \abs{F_m\pth{\fract{1}{2}+it}} dt \ll \frac{\delta x^\theta}{x^{1/2}} \sum_{1\leq m \leq M_1} \int_{0}^{T_1} \abs{F_m\pth{\fract{1}{2}+it}} dt,
\]
the last inequality following from (\ref{eq:LSymmetry}). We divide the ranges of $m$ and $t$ into dyadic intervals $[M,2M]$ and $[T,2T]$ for $M,T\geq 1$ along with the additional interval $[0,1]$ for $t$. Theorem \ref{thm:Asymptotic} now follows from
\begin{lem}\label{lem:MainLemma}
We have
\[
\sum_{M \leq m \leq 2M} \int_{T}^{2T} \abs{F_m\pth{\fract{1}{2}+it}} dt \ll \frac{x^{1/2}}{(\log x)^{2J+3}},
\]
uniformly for $1\leq M \leq M_1$ and $1 \leq T \leq T_1$. The expression with an integral over $[0,1]$ also satisfies this bound.
\end{lem}

\section{Reduction to Large Values}\label{sec:Reduction}

In this section, we reduce the proof of Lemma \ref{lem:MainLemma} to the estimation of the number of large values of a certain Dirichlet polynomials. We begin by letting $\Delta$ be a small parameter to be chosen and write $F_m(s) = G_m(s)H_m(s)$, where $H_m(s)$ is the product of those factors for which the lengths $N_j$ satisfy $N_j \leq x^{\Delta/J}$. Since
\[
\abs{f_{1,m}\pth{\fract{1}{2}+ it}} \ll N_1^{1/2}\log x; \qquad \abs{f_{j,m}\pth{\fract{1}{2}+ it}} \ll N_j^{1/2},\  (j\geq 2),
\]
we  have
\[
\abs{H_m\pth{\fract{1}{2}+it}} \ll Z^{1/2}\log x,
\]
where $Z$ is the product of those $N_j$ with $N_j \leq x^{\Delta/J}$. Then

\begin{equation}\label{eq:BreakApartF}
\int_{T}^{2T} \sum_{M \leq m \leq 2M} \abs{F_m\pth{\fract{1}{2}+it}}dt \ll Z^{1/2}\log x \int_{T}^{2T} \sum_{M \leq m \leq 2M}  \abs{G_m\pth{\fract{1}{2}+it}}dt.
\end{equation}
We now bound the integral on the right ($I$, say) by a set of $O(T)$ well-spaced points $t_n$. We have
\[
I \ll \sum_{n} \sum_{M \leq m \leq 2M} \abs{G_m\pth{\fract{1}{2}+it_n}},
\]
where $\abs{t_l-t_n} \geq 1$ for $l\neq n$. For each triple $j,m,n$, let
\[
\abs{f_{j,m}\pth{\fract{1}{2}+it_n}} = N_j^{\sigma(j,m,n)-1/2}(\log x)^4.
\]
We need to show that $\sigma(j,m,n)$ cannot be too close to 1.  We treat the case $j > J$, for which
\[
f_{j,m}(s) = \sum_{\ida} \frac{\mu(\ida)\lambda^m(\ida)W_j(\ida)}{(N\ida)^s}.
\]
The case $j \leq J$ would be very similar. By Mellin inversion
\[
f_{j,m}\pth{\fract{1}{2}+it} = \int\limits_{(c)} L\pth{\fract{1}{2}+it+s,\lambda^m}^{-1} N_j^s \tilde{W}_j(s) ds,
\]
where $c = \frac{1}{2} + (\log x)^{-1}$. We have trivially that
\[
\frac{1}{\abs{L(1+(\log x)^{-1}+it,\lambda^m)}} \leq \zeta_K(1+(\log x)^{-1}) \ll \log x,
\]
(here again $\zeta_K$ is the Dedekind zeta function for $\Z[i]$). Truncating the integral at height $x^\ep$ and using the rapid decay of $\tilde{W}$ gives
\[
f_{j,m}\pth{\fract{1}{2}+it} = \int_{c-ix^\ep}^{c+ix^\ep} L\pth{\fract{1}{2}+it+s,\lambda^m}^{-1} N_j^s \tilde{W}_j(s) ds
\]
with negligible error. We now use Lemmas \ref{lem:ZeroFreeRegion} and \ref{lem:CritStripEstimates} to move the line of integration to the left of $\Re s = \frac{1}{2}$. Then in the region
\[
1 - \eta \leq \Re w \leq \fract{1}{2} + c, \qquad \abs{\Im w - t} \leq x,
\]
where
\[
\eta = C(\log x)^{-2/3}(\log\log x)^{-1/3},
\]
we have 
\[
\frac{1}{L(w,\lambda^m)} \ll \log x
\]
Moving the line of integration to $1/2 - \eta$, we thus have
\[
\abs{f_{j,m}\pth{\fract{1}{2}+it}} \ll (\log x)N_j^{1/2-\eta},
\]
from which it follows that
\[
\sigma(j,m,n) \leq 1 - \eta
\]
for $x$ sufficiently large. We now split the available range for $\sigma(j,m,n)$ into $O(\log x)$ ranges $I_0 = (-\infty, \frac{1}{2})$ and 
\[
I_l = \hor{\frac{1}{2}+\frac{l-1}{L}, \frac{1}{2}  + \frac{l}{L}},\quad \pth{1 \leq l \leq 1+L/2,\ L = \floor{\log x}},
\]
For each $j,l$, let 
\[
C(j,l) = \set{(m,t_n) : \max_{1\leq k \leq 2J} \sigma(k,m,n)= \sigma(j,m,n)\ \text{and}\  \sigma(j,m,n)\in I_l}.
\]
Since there are $O(\log x)$ classes $C(j,l)$, there must exist some class $\mathcal{C}$ for which
\[
I \ll (\log x)\sum_{(m,t) \in \mathcal{C}} \abs{G_m\pth{\fract{1}{2} + it}}.
\]
For $(m,t) \in \mathcal{C}$, we have
\[
\abs{G_m\pth{\fract{1}{2} + it}} = \prod N_j^{\sigma(j,m,n) - 1/2} \leq \prod N_j^{l/L} = Y^{l/L},
\]
where $Y$ is the product of the $N_j$ with $N_j > x^{\Delta/J}$. To simplify notation, let
\[
\sigma = \frac{1}{2} + \frac{l-1}{L}, \quad f_m(s) = f_{j,m}(s), \quad N = N_j, \quad R = \# \mathcal{C}.
\]
If $l = 0$, then $I \ll MT \log x$, so (\ref{eq:BreakApartF}) gives
\[
\int_{T}^{2T} \sum_{M\leq m \leq 2M} \abs{F_m\pth{\fract{1}{2}+it}} dt \ll Z^{1/2} M_1T_1 (\log x)^2 \ll \delta^{-1}x^{1-\theta+\Delta + \ep} \ll \frac{x^{1/2}}{(\log x)^A}
\]
since we may assume $\Delta < \frac{1}{5}$ and $x^{\theta}\delta > x^{7/10}$. If $l \geq 1$, we have
\[
I \ll (Y^{\sigma - 1/2})R \log x,
\]
and so
\[
\int_{T}^{2T} \sum_{M \leq m \leq 2M} \abs{F_m\pth{\fract{1}{2}+it}}dt \ll Z^{1/2}Y^{\sigma-1/2}R(\log x)^2.
\]
Now since
\[
Z^{1/2}Y^{\sigma-1/2} \ll Z^{1/2}(xZ^{-1})^{\sigma - 1/2} \ll x^{1/2}(Zx^{-1})^{1-\sigma} \ll x^{1/2+(2\Delta-1)(1-\sigma)},
\]
we find that
\begin{equation}\label{eq:MeanValueRBound}
\int_{T}^{2T} \sum_{M \leq m \leq 2M} \abs{F_m\pth{\fract{1}{2}+it}}dt \ll x^{1/2}(\log x)^2 \fracp{R}{x^{(1-2\Delta)(1-\sigma)}}.
\end{equation}
It remains to estimate $R$. For each $(t,m) \in \colc$, we have
\[
\abs{f_m\pth{\fract{1}{2} + it}} \gg N^{\sigma - 1/2},
\]
Since $\sigma \leq 1-\eta/2$ we see that Lemma \ref{lem:MainLemma} follows from the bound
\begin{equation}\label{eq:RBound}
R \ll x^{(1-3\Delta)(1-\sigma)}(\log x)^B
\end{equation}
for any fixed $B > 0$, since then the expression on the right of(\ref{eq:MeanValueRBound}) is bounded by taking $\sigma = 1-\eta/2$, and the definition of $\eta$ allows us to save arbitrary powers of $\log x$. To deduce the requisite bound for $R$, it is sufficient to show that
\begin{equation}\label{eq:RMTBound}
R \ll (MT)^{10(1-\sigma)/3}(\log x)^B
\end{equation}
uniformly in $M,T,\sigma$, since $MT \leq M_1T_1 = x^{1-\theta+\ep}\delta^{-1} \leq x^{3/10-\ep}$.

\section{Mean and Large Value Results}\label{sec:MeanAndLarge}

To estimate $R$, we will need several mean-value results of the form
\begin{equation}\label{eq:MeanValueForm}
\sum_{\abs{m} \leq M} \int_{-T}^{T} \sumabs{\sum_{N\ida\asymp N} c(\ida) \lambda^m(\ida) (N\ida)^{-it}}^2 dt \ll \cold \sum_{N\ida\asymp N} \abs{c(\ida)}^2
\end{equation}
for some $\cold = \cold(N,M,T)$, where $c(\ida)$ are arbitrary complex coefficients defined on the ideals of $\Z[i]$. First, we have Coleman's hybrid large sieve (Theorem 6.2 of \citep{ColemanBinary}).

\begin{lem}[Coleman]\label{lem:ColemanSieve}
The estimate (\ref{eq:MeanValueForm}) holds with
\begin{equation}\label{eq:ColemanSieve}
\cold = M^2 + T^2 + N.
\end{equation}
\end{lem}

\noindent Additionally, we also have the following trivial estimate.
\begin{lem}
The estimate (\ref{eq:MeanValueForm}) holds with
\begin{equation}\label{eq:TrivialSieves}
\cold = MT + N\min(M,T).
\end{equation}
\end{lem}

\begin{proof}
For the case $T\leq M$, see \citep{RicciThesis}, Theorem C. For the other case, the mean-value theorem for Dirichlet polynomials gives
\[
\int_{-T}^{T} \sumabs{\sum_{N\ida\asymp N} c(\ida) \lambda^m(\ida) (N\ida)^{-it}}^2 dt = (T+O(N)) \sum_{N\ida \asymp N} \abs{c(\ida)}^2.
\]
Summing over $m$ gives the other estimate.

\end{proof}

Note that in each of the estimates above, the integral over $t$ can be replaced by a sum over well-spaced points at the cost of a logarithmic factor, which will not affect our results.

For the problem at hand, the natural quantity to work with is $MT$, rather than the minimum or maximum of $M$ and $T$. To this end, let
\[
\coll = \coll(M,T) = \frac{\abs{\log(M/T)}}{\log MT}
\]
so that
\[
\max(M^2,T^2) = (MT)^{1+\coll} \mand \min(M^2,T^2) = (MT)^{1-\coll}.
\]
We will regard $\coll$ as an arbitrary parameter assuming values in $[0,1]$. The estimates (\ref{eq:ColemanSieve}) and (\ref{eq:TrivialSieves}) become, respectively
\[
(MT)^{1+\coll} + N \mand MT + N(MT)^{(1-\coll)/2}.
\]
We will apply these estimates to suitable powers of the polynomial $f_m\pth{\fract{1}{2}+it}$. For any integer $g \geq 1$, we have
\[
R N^{g(2\sigma -1)} \ll \sum_{(m,t) \in\colc} \sumabs{\sum_{\ida} \frac{a(\ida)W(N\ida)\lambda^m(\ida)}{(N\ida)^{1/2+it}}}^{2g} \ll \cold(N^g,M,T) \sum_{N\ida\asymp N} \frac{\abs{b(\ida)}^2}{N\ida},
\]
say where $\abs{b(\ida)} \leq d_g(\ida)(\log x)^g$. Using Lemma \ref{lem:DivisorBound} and partial summation, we find that the coefficient sum on the right is $O((\log x)^B)$ for some $B$ which depends on $g$. Since $g$ is bounded in terms of $\Delta$, we find that $B$ and the implied constant depend at most on our choice of $\Delta$. Thus
\[
RN^{g(2\sigma-1)} \ll \pth{MT + N^g(MT)^{(1-\coll)/2}} (\log x)^B,
\]
\[
RN^{g(2\sigma-1)} \ll \pth{(MT)^{1+\coll} + N^g}(\log x)^B.
\]
We will also make use of the following large values result of Coleman (Theorem 7.3 of \citep{ColemanBinary}  with $\theta=0$) which is proved using Huxley's subdivision method:
\[
R \ll \pth{N^{2g(1-\sigma)}+(M^2+T^2)N^{g(4-6\sigma)} }(\log x)^B \ll \pth{N^{2g(1-\sigma)}+(MT)^{1+\coll} N^{g(4-6\sigma)} }(\log x)^B.
\]
For any integer $g \geq 1$, the estimates above give

\begin{equation}\label{eq:R-Coleman}
R\ll \pth{(MT)^{1+\coll}N^{g(1-2\sigma)} + N^{2g(1-\sigma)}}(\log x)^B,
\end{equation}
\begin{equation}\label{eq:R-Trivial}
R\ll \pth{MT N^{g(1-2\sigma)} + (MT)^{(1-\coll)/2} N^{2g(1-\sigma)}}(\log x)^B,
\end{equation}
\begin{equation}\label{eq:R-Huxley}
R\ll \pth{(MT)^{1+\coll}N^{g(4-6\sigma)} + N^{2g(1-\sigma)}}(\log x)^B.
\end{equation}

The last estimate is useful only when $\sigma \geq 3/4$, and any time it is used, $\sigma$ will be assumed to lie in this range. In the each of the estimates above, the first summand decreases in $g$, and the second increases. Writing $N = (MT)^\beta$, one would like to choose
\begin{equation}\label{eq:idealgchoice}
g = \frac{1+\coll}{\beta}, \quad \frac{1+\coll}{2\beta}, \quad \frac{1+\coll}{2\beta(2\sigma-1)},
\end{equation}
respectively, so as to equalize the two summands in each estimate. 

%
%
Unfortunately, $g$ must be chosen to be an integer, and this adds a fair amount of complication to our analysis. The optimal choices for $g$ in (\ref{eq:R-Coleman}) -- (\ref{eq:R-Huxley}) are obtained by taking the floor of the values in (\ref{eq:idealgchoice}), or 1 plus the floor. Thus, unconditionally, we have $R \ll (MT)^{\min(\alpha_1,\ldots,\alpha_6)}$, where
\begin{equation}\label{eq:allEstimates}
\begin{aligned}
\alpha_1(\coll,\beta,\sigma) &= 1 + \coll + \beta \floor{\frac{1+\coll}{\beta}}(1-2\sigma), \\
\alpha_2(\coll,\beta,\sigma) &= 2\beta\floor{\frac{1+\coll}{\beta}+1}(1-\sigma), \\
\alpha_3(\coll,\beta,\sigma) &= 1 + \beta \floor{\frac{1+\coll}{2\beta}}(1-2\sigma), \\
\alpha_4(\coll,\beta,\sigma) &= \frac{1-\coll}{2} + 2\beta \floor{\frac{1+\coll}{2\beta} + 1}(1-\sigma), \\
\alpha_5(\coll,\beta,\sigma) &= 1+\coll + \beta \floor{\frac{1+\coll}{2\beta(2\sigma-1)}}(4-6\sigma), \\
\alpha_6(\coll,\beta,\sigma) &= 2\beta \floor{\frac{1+\coll}{2\beta(2\sigma-1)} + 1}(1-\sigma).
\end{aligned}
\end{equation}
where $\alpha_1,\alpha_3,\alpha_5$ apply only when the expression in the floor brackets is at least 1. We also define $\cola_0$ to be the minimum of these six estimates and $\cola_i = \min(\alpha_{2i-1},\alpha_{2i})$ for $i=1,2,3$. Our analysis now proceeds by fixing $\beta$ and $\sigma$ and understanding the behavior of $\cola_0$ as $\coll$ ranges between $0$ and $1$. For this, we will need the following propositions which describe the behavior of $\cola_i$ for $i=1,2,3$. The proofs of these propositions are very similar, so we only prove Proposition \ref{prop:a1}. For notational brevity, we also suppress the dependence of $\alpha_i$ and $\cola_i$ on $\beta$ and $\sigma$. 

\begin{prop}\label{prop:a1}
Fix $\beta \in (0,\frac{5}{6})$ and $\sigma \in (\frac{7}{10},\frac{3}{4})$. For $n\in\Z$, define
\[
\coll_n^{(1,d)} = \beta n-1 \mand \coll_n^{(1,e)} = \coll_n^{(1,d)} + 2\beta(1-\sigma).
\]
Then on $[0,1]\cap\hor{\coll_{n}^{(1,d)}, \coll_{n+1}^{(1,d)}}$, we have
\[
\cola_1(\coll) = \begin{cases}
1+\coll + \beta n (1-2\sigma) & \text{if $\coll \leq \coll_{n}^{(1,e)}$}, \\
2\beta(n+1)(1-\sigma) & \text{if $\coll \geq \coll_{n}^{(1,e)}$}.
\end{cases}
\]
In particular, $\cola_1(\coll)$ is a continuous non-decreasing function of $\coll$ on $[0,1]$.
\end{prop}

\begin{prop}\label{prop:a2}
Fix $\beta\in(0,\frac{5}{3})$ and $\sigma \in (\frac{7}{10},1)$. For $n\in\Z$, define
\[
\coll_n^{(2,d)} = 2\beta n-1 \mand \coll_n^{(2,e)} = \coll_n^{(2,d)}  + 4\beta(1-\sigma).
\]
Then on $[0,1]\cap\hor{\coll_{n}^{(2,d)}, \coll_{n+1}^{(2,d)}}$, we have
\[
\cola_2(\coll) = \begin{cases}
1+\beta n (1-2\sigma) & \text{if $\coll \leq \coll_{n}^{(2,e)}$}, \\
\frac{1-\coll}{2}+2\beta (n+1) (1-\sigma) & \text{if $\coll \geq \coll_{n}^{(2,e)}$},
\end{cases}
\]
In particular, $\cola_2(\coll)$ is a continuous non-increasing function of $\coll$ on $[0,1]$.
\end{prop}

\begin{prop}\label{prop:a3}
Fix $\beta \in (0,\frac{5}{6})$ and $\sigma \in (\frac{3}{4},1)$. For $n\in\Z$, define
\[
\coll_n^{(3,d)} = 2\beta(2\sigma-1) n -1 \mand \coll_n^{(3,e)} = \coll_n^{(3,d)} + 2\beta(1-\sigma).
\]
Then on $[0,1]\cap\hor{\coll_{n}^{(3,d)}, \coll_{n+1}^{(3,d)}}$, we have
\[
\cola_3(\coll) = \begin{cases}
1+\coll + \beta n (4-6\sigma) & \text{if $\coll \leq \coll_{n}^{(3,e)}$}, \\
2\beta (n+1) (1-\sigma) & \text{if $\coll \geq \coll_{n}^{(3,e)}$},
\end{cases}
\]
In particular, $\cola_3(\coll)$ is a continuous non-decreasing function of $\coll$ on $[0,1]$.
\end{prop}

\begin{proof}[Proof of Proposition \ref{prop:a1}] A short computation shows that the solutions to $\alpha_1(\coll) = \alpha_2(\coll)$ are given by 
\[
\coll_m^{(1,e)} = m\beta - 1 +2\beta(1-\sigma), \qquad \text{for}\ \frac{1}{\beta}-2(1-\sigma) \leq m \leq \frac{2}{\beta} - 2(1-\sigma),
\]
and that the points of discontinuity of $\cola_1(\coll)$ are given by
\[
\coll_n^{(1,d)} = n\beta-1, \qquad \frac{1}{\beta} \leq n \leq \frac{2}{\beta}.
\]
Since $\coll_{m+1}^{(1,e)} - \coll_m^{(1,e)} = \coll_{n+1}^{(1,d)} - \coll_{n}^{(1,d)} = \beta$, and since $\sigma \neq 1$, there is a unique point of intersection, say $\coll_{m_n}^{(1,e)}$, between each pair $\coll_{n}^{(1,d)}$, $\coll_{n+1}^{(1,d)}$ of points  of discontinuity, and it is easy to check that $m_n = n$. Moreover, for a fixed value of $\floor{\frac{1+\coll}{\beta}}$, i.e. on the interval between two points of discontinuity, it is clear that $\alpha_1$ increases in $\coll$, and $\alpha_2$ is constant. Thus $\cola_1$ is non-decreasing on each interval $\hor{\coll_{n}^{(1,d)}, \coll_{n+1}^{(1,d)}}$. Finally, we note that $\alpha_2\pth{\coll_{n-1}^{(1,d)}} = \alpha_2\pth{\coll_{n}^{(1,d)}-\ep} = \alpha_1\pth{\coll_{n}^{(1,d)}}$ for all $\ep > 0$ sufficiently small. Thus $\cola_1$ is continuous, proving the last statement of the proposition.

\end{proof}

It is worth noting that the results of these propositions extend to some slightly wider ranges of $\beta$ and $\sigma$. For clarity of exposition, we have included only the ranges we need for our analysis. From these propositions, we can also deduce the following upper bounds, which have the benefit of being linear in $\coll$.

\begin{cor}\label{cor:LinearUpperBounds}
For all $\coll\in[0,1]$, $\beta \in (0,\frac{2}{3})$, and $\sigma \in \pth{\fract{7}{10},1}$, we have $\cola_i \leq \colb_i$, where
\[
\begin{aligned}
\colb_1(\coll,\beta,\sigma) &= 2(1+\coll+\beta)(1-\sigma) - 4\beta(1-\sigma)^2,\\
\colb_2(\coll,\beta,\sigma) &= \pth{\fract{1}{2}-\sigma}(1+\coll-2\beta(1-\sigma)) + 2\beta(1-\sigma),\\
\colb_3(\coll,\beta,\sigma) &= \fracp{1-\sigma}{2\sigma-1}(1+\coll - 2\beta(1-\sigma)) + 2\beta(1-\sigma).
\end{aligned}
\]
\end{cor}

\begin{proof}
The functions $\colb_i$ are the linear interpolations of the points $\pth{\coll_m^{(i,e)}, \cola_i\pth{\coll_m^{(i,e)}}}$.
\end{proof}

\section{Short Polynomials}\label{sec:ShortPolynomials}

We are now ready to apply the estimates in Section \ref{sec:MeanAndLarge} to estimate the quantity $R$. We will need a subconvexity estimate for Hecke $L$-functions (Lemma \ref{lem:Subconvexity} below) to eliminate certain ranges of $\coll,\beta,\sigma$. This will require the coefficients $a(\ida)$ to be smooth, which is ensured by $N > X$. As such, the present section is devoted to the case $N \leq X$, where we do not require subconvexity. We divide into several cases.

\subsection*{Case 1.1: $MT \leq X$}
Choose $g$ so that
\[
X^2 \leq N^g \leq X^3.
\]
Then $(MT)^{1+\coll} \leq (MT)^2 \leq X^2$, so by (\ref{eq:R-Coleman}), we have
\[
R \ll \pth{X^{2+2(1-2\sigma)} + X^{6(1-\sigma)}}(\log x)^B \ll x^{6(1-\sigma)/J}(\log x)^B.
\]
This gives (\ref{eq:RBound}) so long as $J > 6$ and $\Delta$ is sufficiently small. 

\subsection*{Case 1.2: $MT > X$, $\beta > \frac{2}{3}$}
In this case, we have $\beta < 1$. If $\beta > \frac{2}{3}$, then $(MT)^{1+\coll} \leq X^3$. Similar to Case 1, we choose $g$ so that
\[
X^3 \leq N^g \leq X^4
\]
and apply (\ref{eq:R-Coleman}) to obtain
\[
R \ll x^{8(1-\sigma)/J}(\log x)^B.
\]
We obtain (\ref{eq:RBound}) so long as $J > 8$ and $\Delta$ is sufficiently small. 

\subsection*{Case 1.3: $MT > X$, $\beta \leq \frac{2}{3}$, $\sigma \leq \frac{3}{4}$}

Here it suffices to use the estimates $\colb_1$ and $\colb_2$. A short computations shows that $\colb_1(\coll) \leq \frac{10}{3}(1-\sigma)$ so long as
\[
\coll \leq \pth{\frac{2}{3} - \beta(2\sigma-1)}(1-\sigma) = \coll^*,
\]
say. Since $\colb_2$ decreases in $\coll$, it suffices to check that $\colb_2(\coll^*) \leq \frac{10}{3}(1-\sigma)$. Another computation shows that this inequality holds so long as
\[
\beta \leq \frac{10\sigma-9}{12 \sigma^2-24 \sigma+9}.
\]
The expression on the right decreases in $\sigma$, and substituting $\sigma = \frac{3}{4}$, we see that $\colb_2(\coll^*) \leq \frac{10}{3}(1-\sigma)$ so long as $\beta \leq \frac{2}{3}$.

\subsection*{Case 1.4: $MT > X$, $\beta \leq \frac{2}{3}$, $\sigma > \frac{3}{4}$}

The proof of this case is very similar to Case 3, except that we use $\cola_3$ in place of $\cola_1$. Note that Cases 3 and 4 do not use any information about the size of $MT$ compared to $X$. As such, Cases 3 and 4 actually cover the entire range $\coll \in [0,1]$, $\beta \leq 2/3$, $\sigma \in (\frac{7}{10},1)$. 

\section{Long Polynomials: Subconvexity and Simplifications}\label{sec:LongPolynomials1}

We now suppose that $N > X$, in which case we may apply the following 
subconvexity estimate for Hecke $L$-functions.

\begin{lem}[Ricci]\label{lem:Subconvexity}
If $(4m^2+t^2) \geq 4$, then
\[
L\pth{\fract{1}{2}+it,\lambda^m} \ll (m^2+t^2)^{1/6}\log^3(m^2+t^2).
\]
\end{lem}

For a proof, see \citep{RicciThesis}, Chapter 2, Theorem 4. Since $N > X$, the coefficients of $f_m(1/2+it)$ are smooth and we may write (in the case $j > 1$)
\begin{equation}\label{eq:fmMellinInvert}
f_m\pth{\fract{1}{2}+it} = \frac{1}{2\pi i} \int\limits_{(0)} L\pth{\fract{1}{2}+it+s,\lambda^m} \tilde{W}(s) N^s ds.
\end{equation}
We have $m \geq 1$ always, so Lemma \ref{lem:Subconvexity} yields
\[
\begin{aligned}
f_m\pth{\fract{1}{2}+it} &\ll \int_{-\infty}^{\infty} \frac{\pth{m^2+t^2+y^2}^{1/6}\log^3(m^2+t^2+y^2)}{(1+\abs{y})^A}\ dy \\
&\ll (M^2+T^2)^{1/6}\log^3(M^2+T^2) \\
&\ll (MT)^{(1+\coll)/6}\log^3(M^2+T^2).
\end{aligned}
\]
If $j=1$, we write
\[
f_m\pth{\fract{1}{2}+it} = \log N \sum_{\ida} W\fracp{N\ida}{N} \frac{\lambda^m(\ida)}{(N\ida)^{1/2+it}} + \sum_{\ida} W\fracp{N\ida}{N}\log\fracp{N\ida}{N} \frac{\lambda^m(\ida)}{(N\ida)^{1/2+it}}.
\]
The first sum is handled in the same way as before. If $W(y)$ is replaced by $W^*(y)=W(y)\log y$, then $\tilde{W}^*$ decays rapidly on vertical lines just as $\tilde{W}$, and so in this case we obtain
\[
f_m\pth{\fract{1}{2}+it} \ll (MT)^{(1+\coll)/6}\log^3(M^2+T^2) \log N.
\]
Since $\abs{f_m\pth{\fract{1}{2}+it}} \gg (MT)^{\beta(\sigma-1/2)}(\log x)^4$, we deduce that
\begin{equation}\label{eq:Subconvexity}
\sigma \leq \frac{1}{2} + \frac{1+\coll}{6\beta}.
\end{equation}
We can also make a few simplifying assumptions. We may assume
\[
\sigma > \frac{7}{10},
\]
for otherwise
\[
R \leq MT \leq (MT)^{10(1-\sigma)/3}.
\]
In particular, we have $\sigma \leq \frac{7}{10}$ if $\beta \geq \frac{5}{6}(1+\coll)$. From the remarks in Case 4, we may also assume $\beta > \frac{2}{3}$. Thus we may limit our analysis to the situation in which
\begin{equation}\label{eq:simplifiedBeta}
\frac{2}{3} < \beta < \frac{5}{6}(1+\coll) \leq \frac{5}{3}.
\end{equation}

\section{Long Polynomials: Case Checking}\label{sec:LongPolynomials2}

\subsection*{Case 2.1: $\beta \leq \frac{5}{6}$, $\sigma \leq \frac{3}{4}$}

Fix $\sigma \in (\frac{7}{10},\frac{3}{4}]$ and $\beta \in (\frac{2}{3},\frac{5}{6}]$. We determine the largest value $\coll^*$ of $\coll$ for which $\cola_1$ is sufficient. Since $\cola_1$ is continuous and non-decreasing, we can compute $\coll^*$ as follows. We have 
\[
2\beta(n+1)(1-\sigma) \leq \frac{10}{3}(1-\sigma)
\]
so long as  $n \leq \frac{5}{3\beta}-1$. If $\beta\neq \frac{5}{6}$, then since $\floor{\frac{5}{3\beta}} = 2$ in the present case, it follows that $\coll^*$ lies in the interval $\hor{\coll_{2}^{(1,d)},\coll_{2}^{(1,e)}}$. The value $\coll^*$ is then given by the solution to
\[
1 + \coll^* - 2\beta (2\sigma-1) = \frac{10}{3}(1-\sigma).
\]
If $\beta =\frac{5}{6}$, then $\cola_1(\coll) = \frac{10}{3}(1-\sigma)$ for all $\coll \in \left[\coll_{1}^{(1,e)}, \coll_{2}^{(1,d)} \right]$, so we may take $\coll^* = \coll_2^{(1,d)} = \frac{2}{3}$. Thus in this case also, $\coll^*$ is given by the solution to the equation above.

We have $\cola_1(\coll) \leq \frac{10}{3}(1-\sigma)$ so long as $\coll \leq \coll^*$. Since $\cola_2$ is continuous and non-increasing, to estimate the remaining range of $\coll$, it suffices to check that $\cola_2(\coll^*) \leq \frac{10}{3}(1-\sigma)$. To evaluate $\cola_2(\coll^*)$, we need to determine $n^*$ such that the interval $\hor{\coll_{n^*}^{(2,d)}, \coll_{n^*+1}^{(2,d)}}$ contains $\coll^*$. A short computation shows that in the present case, we have
\[
n^* = \floor{\frac{5}{3\beta}(1-\sigma) + (2\sigma-1)} = 1.
\]
If $\coll^* \leq \coll_{1}^{(2,e)}$, then
\[
\cola_2(\coll^*) = 1+\beta (1-2\sigma)  = 1-\beta + 2\beta(1-\sigma)  \leq \frac{10}{3}(1-\sigma),
\]
where the last inequality follows from $1-\beta \leq \frac{1}{6} < \frac{5}{12} \leq \frac{5}{3}(1-\sigma)$. Otherwise if $\coll^* > \coll_{1}^{(2,e)}$, then again we have
\[
\cola_2(\coll^*) = \frac{1-\coll^*}{2} + 4\beta(1-\sigma) \leq \frac{1-\coll_{1}^{(2,e)}}{2} + 4\beta(1-\sigma) = 1-\beta + 2\beta(1-\sigma) \leq \frac{10}{3}(1-\sigma).
\]

\subsection*{Case 2.2: $\beta \leq \frac{5}{6}$, $\sigma >\frac{3}{4}$}

Fix $\sigma \in (\frac{3}{4},1)$ and $\beta \in (\frac{2}{3},\frac{5}{6}]$. The arguments for this case and the next are very similar to Case 2.1, so we will be fairly brief. As in Case 2.1, we determine the largest value $\coll^*$ of $\coll$ for which $\cola_3$ is sufficient. Arguing as in that case, we find that  $\coll^*$ is given by the solution to
\[
1 + \coll^* - 2\beta (6\sigma-4) = \frac{10}{3}(1-\sigma).
\]
We now check that $\cola_2(\coll^*) \leq \frac{10}{3}(1-\sigma)$. As before, we have $\coll^* \in \hor{\coll_{1}^{(2,d)}, \coll_{2}^{(2,d)}}$. If $\coll^* \leq \coll_{1}^{(2,e)}$, then
\[
\cola_2(\coll^*) = 1+\beta (1-2\sigma)  = 1-\beta + 2\beta(1-\sigma)  \leq \frac{10}{3}(1-\sigma)
\]
so long as $\sigma \leq \frac{9}{10}$, where the last inequality follows from $1-\beta \leq \frac{1}{6}  \leq \frac{5}{3}(1-\sigma)$. If $\sigma > \frac{9}{10}$, then $\coll^* > 1$, so $\cola_3$ suffices for all $\coll\in[0,1]$. If $\coll^* > \coll_{1}^{(2,e)}$, then just as in Case 2.1 we have
\[
\cola_2(\coll^*)  \leq \frac{10}{3}(1-\sigma).
\]

\subsection*{Case 2.3: $\beta > \frac{5}{6}$}

Fix $\sigma \in (\frac{7}{10},1)$ and $\beta \in (\frac{5}{6},\frac{5}{3})$. By the subconvexity restriction (\ref{eq:Subconvexity}), we may assume $\coll > 3\beta(2\sigma-1)-1 = \coll^*$, say, and since $\cola_2$ is non-increasing in $\coll$, it suffices to check that $\cola_2(\coll^*) \leq \frac{10}{3}(1-\sigma)$. The inequalities $\coll_{0}^{(2,e)} \leq \coll^* \leq \coll_1^{(2,d)}$ are easy to verify (the interval $\left[\coll_{0}^{(2,e)}, \coll_1^{(2,d)}\right]$ may intersect only part of $[0,1]$, but this is immaterial). It follows that
\[
\begin{aligned}
\cola_2(\coll^*) &= \frac{1-\pth{3\beta(2\sigma-1)-1}}{2} + 2\beta(1-\sigma) = 1 + 3\beta(\fract{1}{2}-\sigma) + 2\beta(1-\sigma) \\
&= 1 - \frac{3\beta}{2} + 5\beta(1-\sigma) \leq \frac{10}{3}(1-\sigma)\pth{1 - \frac{3\beta}{2}} + 5\beta(1-\sigma) = \frac{10}{3}(1-\sigma).
\end{aligned}
\]

\section{Optimality of $\frac{10}{3}$}\label{sec:Conclusion}

There are two sets of values of $\coll,\beta,\sigma$ which show that the constant $\frac{10}{3}$ is optimal in our analysis. These are
\[
\coll=\frac{3}{5},\ \beta=\frac{4}{3},\ \sigma=\frac{7}{10} \mand \coll=1,\ \beta=\frac{5}{6},\ \sigma=\frac{9}{10}.
\]
In the cases above where these values occur, one may check that the inequalities used are sharp, and so $\frac{10}{3}$ cannot be improved. The optimal large sieve (\ref{eq:OptimalLargeSieve}) would eliminate the need for the variable $\coll$, but the particular case $\beta=\frac{5}{6}, \sigma=\frac{9}{10}$ remains a worst case when using this estimate.

\bibliography{references}

\end{document}